\documentclass[11pt,leqno,oneside]{amsart}
\usepackage{layout}
\usepackage{mathrsfs,dsfont}
\usepackage{amsmath,amstext,amsthm,amssymb,bbm,color}
\usepackage{comment}
\usepackage{charter}
\usepackage{typearea}
\usepackage{pdfsync}
\usepackage[width=6.5in,height=8.7in]{geometry}

\usepackage{amsfonts}

\usepackage[utf8]{inputenc}

\def\bt{\begin{thm}}
\def\et{\end{thm}}
\def\bl{\begin{lem}}
\def\el{\end{lem}}
\def\bd{\begin{defn}}
\def\ed{\end{defn}}
\def\bc{\begin{cor}}
\def\ec{\end{cor}}
\def\bp{\begin{proof}}
\def\ep{\end{proof}}
\def\br{\begin{rem}}
\def\er{\end{rem}}

\newtheorem{thm}{Theorem}[section]
\newtheorem{prop}[thm]{Proposition}
\newtheorem{lem}[thm]{Lemma}
\newtheorem{defn}[thm]{Definition}

\newtheorem{rem}[thm]{Remark}
\newtheorem{cor}[thm]{Corollary}

\numberwithin{equation}{section}

\newcommand{\bthm}{\begin{thm}}
\newcommand{\ethm}{\end{thm}}
\newcommand{\bstp}{\begin{stp}}
\newcommand{\estp}{\end{stp}}
\newcommand{\blemma}{\begin{lemma}}
\newcommand{\elemma}{\end{lemma}}
\newcommand{\bprop}{\begin{prop}}
\newcommand{\eprop}{\end{prop}}
\newcommand{\bpf}{\begin{pf}}
\newcommand{\epf}{\end{pf}}
\newcommand{\bdefn}{\begin{defn}}
\newcommand{\edefn}{\end{defn}}
\newcommand{\brk}{\begin{rmrk}}
\newcommand{\erk}{\end{rmrk}}
\newcommand{\bcrl}{\begin{crl}}
\newcommand{\ecrl}{\end{crl}}

\usepackage{amsfonts}

\title[]{Asymptotic mass distribution  of random holomorphic sections}

\address{}
\address{Faculty of Engineering and Natural Sciences, Sabanc{\i} University, \.{I}stanbul, Turkey}
\email{tbayraktar@sabanciuniv.edu}
\email{afrimbojnik@sabanciuniv.edu}

\date{\today}


\begin{document}

\author{Turgay Bayraktar \and Afrim Bojnik}
\thanks{T.\ Bayraktar is partially supported by T\"{U}B\.{I}TAK \& German DAAD Collaboration Grant ARDEB-2531/121N191 and T\"{U}B\.{I}TAK 2219 Grant}
\thanks{A.\ Bojnik is partially supported by Tosun Terzioğlu Chair Postdoctoral Fellowship}

\begin{abstract}
In this note, we prove a central limit theorem for the mass distribution of random holomorphic sections associated with a sequence of positive line bundles endowed with $\mathscr{C}^3$ Hermitian metrics over a compact K\"{a}hler manifold. In addition, we show that almost every sequence of such random holomorphic sections exhibits quantum ergodicity in the sense of Zelditch.

  \end{abstract}
\maketitle


\section{Introduction}

The study of asymptotic mass distributions has emerged as a fundamental topic in complex geometry and quantum chaos, providing a bridge between geometric structures and probabilistic phenomena. Mass distributions, derived from normalized $L^2$-densities of holomorphic sections of positive line bundles over compact K\"{a}hler manifolds, offer critical insights into equidistribution, quantum ergodicity, and the statistical properties of zeros of holomorphic sections. These measures encapsulate the interaction between the geometry of the underlying manifold and the probabilistic behavior of random holomorphic sections, making them a central object of study in several mathematical and physical contexts.

Let $(X, \omega)$ be a compact K\"{a}hler manifold of dimension $n$ and $(L, h)$ a positive Hermitian holomorphic line bundle over $X$ with curvature form satisfying the prequantum condition $c_1(L, h) = \omega$.  For a positive integer $p$, the $p$-th tensor power $L^p:=L^{\otimes p}$ of $L$ is equipped with the induced  Hermitian metric $h^p$.
 The \emph{mass distribution} of a holomorphic section $s_p \in H^0(X, L^p)$ is defined as the measure
\begin{equation*}
\mathcal{M}_p(s_p): = \frac{1}{p^n} |s_p(x)|^2_{h^p} dV_X,
\end{equation*}
where $|s_p(x)|^2_{h^p}$ denotes the pointwise norm of $s_p$ with respect to $h^p$, and $dV_X = \omega^n / n!$ is the volume form induced by the K\"{a}hler metric.
These measures provide a quantitative description of the spatial distribution of the $L^2$-norm density of holomorphic sections over $X$. We remark that the normalization factor used here differs from the standard choice (cf. \cite{SZ99, Zel18}). However, for our purposes, this normalization is more suitable.

A key concept in this area is \emph{quantum ergodicity}, which examines the equidistribution of eigenfunctions of quantized systems whose classical dynamics are ergodic. In the context of positive line bundles, quantum ergodicity corresponds to the equidistribution of mass measures associated with holomorphic sections. Specifically,  a sequence of holomorphic sections $s_p \in H^0(X, L^p)$ are said to exhibit quantum ergodicity if its mass distributions $\mathcal{M}_p(s_p)$ converge weakly to the volume form $dV_X$ as $p$ grows towards infinity. 

The asymptotics of mass distributions and their connection to zeros of holomorphic sections have been extensively studied in the context of prequantum line bundles. In the seminal work, Shiffman and Zelditch \cite{SZ99} proved that for a sequence of unit-norm holomorphic sections $s_p \in H^0(X, L^p)$, if the mass measures $\mathcal{M}_p(s_p)$ converge weakly to the volume form $dV_X$, the normalized zero currents $\frac{1}{p} [Z_{s_p}]$, corresponding to the zero divisors of $s_p$, converge weakly to the curvature form of the line bundle. Moreover, they showed that this is the case for generic sequences of global holomorphic sections in the probabilistic sense (see \cite[Theorem 1.1]{SZ99}).These results generalize earlier observations by Nonnenmacher and Voros \cite{NV98} for theta bundle over an elliptic curve $\mathbb{C}/\mathbb{Z}^2$. In another direction, Rudnick \cite{Rud05} proved a similar result for modular cusp forms of weight $2p$, corresponding to positive line bundles on non-compact Riemann surfaces. More recently, Zelditch \cite{Zel18} extended the concept of quantum ergodic sequences from positively curved Hermitian metrics to general smooth metrics on ample line bundles, incorporating Bernstein–Markov measures to establish a broader framework for mass distribution asymptotics. For more on zeros of random holomorphic sections and polynomials  we refer to \cite{HAM56, Kac43, LO43, Hann, BBL, SZ99, SZ08, SZ10, BSZ, Ba,BCHM} and references therein.

In parallel, Bayraktar \cite{Bay20} studied asymptotic mass distribution for multivariate random polynomials with sub-Gaussian coefficients. By considering random linear combinations of orthogonal polynomials associated with a smooth weight function of super-logarithmic growth he proved quantum ergodicity in this setting by using Hanson-Wright inequality. More recently, in the framework of Berezin-Toeplitz quantization on complete Hermitian manifolds, a strong law of large numbers was established for mass distribution  of square-integrable Gaussian holomorphic sections in \cite{DLM24}. Unlike quantum ergodicity, their result focuses on averaged behavior rather than the equidistribution of individual sections. All the aforementioned results rely on Bergman kernel asymptotics and potential theory, with the geometric setting primarily involving tensor powers of a single line bundle.

While the equidistribution and average behavior of mass distributions are well-understood for tensor powers, finer statistical properties, such as asymptotic normality, has remained open. This paper aims to bridge this gap by investigating the asymptotic normality of mass distributions and quantum ergodicity for random holomorphic sections associated with a general sequence of holomorphic line bundles $(L_p, h_p)$ over a compact K\"{a}hler manifold, where the curvature forms $c_1(L_p, h_p)$ satisfy the following Diophantine approximation condition
\begin{equation}\label{appr}\frac{1}{A_{p}}c_{1}(L_{p}, h_{p}) = \omega + O(A_{p}^{-a}) \,\,\text{in}\,\,\text{the}\,\,\mathscr{C}^{0}\text{-topology}\,\,\text{as}\,\,p\rightarrow \infty, \end{equation}where $a>0$, $A_{p}>0$ and $\lim_{p\rightarrow \infty}{A_{p}}= \infty$.  The general framework of considering sequences of line bundles $(L_p, h_p)$ with varying curvature, and the analytic tools needed to study their Bergman kernels, was introduced in the seminal paper \cite{CMM}. The special geometric setup involving Diophantine approximation was first formulated and used systematically in \cite{CLMM} where the authors use the $\mathscr{C}^{\infty}$-norm topology induced by the Levi-Civita connection $\nabla^{TX}$ to handle the complete asymptotic expansion of the Bergman kernel function. However, for our purposes, the weaker $\mathscr{C}^{0}$-norm topology suffices.

In this context, we denote by $H^0(X,L_p)$ the space of global holomorphic sections of the line bundle $L_p$. We equip the space of smooth sections $\mathscr{C}^\infty(X, L_p)$ of $L_p$ with the inner product
\begin{equation}\label{innerprod}
    \langle s_1, s_2 \rangle_p := \int_X \langle s_1(x), s_2(x) \rangle_{h_p}\, dV_X(x), \quad \|s\|_p^2 := \langle s, s \rangle_p,
\end{equation}
where $dV_X = \omega^n/n!$ is the volume form induced by the K\"{a}hler metric on $X$. Since $X$ is compact, the Cartan–Serre finiteness theorem (cf. \cite[Chapter~5]{GR04}) implies that $H^0(X,L_p)$ is finite-dimensional, and we denote its dimension by $d_p$. Let  $\{S_j^p\}_{j=1}^{d_p}$ be an orthonormal basis (ONB) of $H^0(X,L_p)$  with respect to the inner product (\ref{innerprod}). A Gaussian random holomorphic section is then defined by
\begin{equation}
s_p = \sum_{j=1}^{d_p} a_j\, S_j^p,
\end{equation}
where the coefficients $a_j$ are i.i.d. standard complex Gaussian random variables. This construction induces a natural probability measure $\mathbb{P}_p$ on $H^0(X,L_p)$; see Section \ref{Section3} for details. In our framework, we define the mass distribution associated with $s_p$ by
\begin{equation}
\mathcal{M}_p(s_p) := \frac{1}{A_p^n}\, |s_p(x)|^2_{h_p}\, dV_X,
\end{equation}
and the corresponding linear statistic is given by integrating a continuous test function against the mass measure, i.e.,
\begin{equation}
\mathcal{M}_p^\phi(s_p) := \frac{1}{A_p^n} \int_X |s_p(x)|^2_{h_p}\, \phi(x)\, dV_X(x),
\end{equation} for any $\phi\in \mathscr{C}^0(X,\mathbb{R})$ where $n=\dim_{\Bbb{C}}X$.
Our main result is the following theorem, which establishes asymptotic normality for these linear statistics as $p \rightarrow \infty$.

\begin{thm}\label{maint}
Let $\{(L_p, h_p)\}_{p \geq 1}$ be a sequence of positive holomorphic line bundles over a compact K\"{a}hler manifold $(X, \omega)$ of complex dimension $n$, satisfying the Diophantine condition~\emph{(\ref{appr})}, and endowed with Hermitian metrics of class~$\mathscr{C}^3$ such that $\|h_p\|_3^{1/3} / \sqrt{A_p} \to 0$ as $p \to \infty$. For each $p$, let $s_p \in H^0(X, L_p)$ be a Gaussian random holomorphic section, and let $\phi \in \mathscr{C}^0(X, \mathbb{R})$. Then the normalized linear statistics
\begin{equation}
    \frac{\mathcal{M}_p^\phi(s_p) - \mathbb{E}[\mathcal{M}_p^\phi(s_p)]}{\sqrt{\mathrm{Var}[\mathcal{M}_p^\phi(s_p)]}}
\end{equation}
converge in distribution to the standard real Gaussian law~$\mathcal{N}_{\mathbb{R}}(0,1)$ as $p \to \infty$.
\end{thm}
We remark that this theorem can be viewed as a form of \emph{geometric universality}, as the limiting Gaussian distribution is independent of the specific geometric constraints of the line bundles. Moreover, it holds in the more general context of sequences of line bundles satisfying the Diophantine approximation condition (\ref{appr}). To the best of our knowledge, this aspect has not been previously studied in the context of mass distributions.

In contrast, the asymptotic normality for smooth linear statistics of the zero divisors of random holomorphic sections has been extensively studied in various settings (cf. \cite{ST, SZ10, Bay16, BG1, BG2, DLM24}). A key ingredient in these studies is the normalized Bergman kernel, which serves as the covariance function for the associated Gaussian holomorphic fields on $X$. This connection allows the problem to be reduced to the foundational result of Sodin and Tsirelson \cite[Theorem 2.2]{ST} on nonlinear functionals of Gaussian processes.
 Our analysis of mass distributions builds on related ideas, relying on refined Bergman kernel asymptotics tailored to the Diophantine approximation condition, variance estimates, and a detailed examination of the associated distribution functions.

Along the way, we also prove an equidistribution result for the mass distributions of random holomorphic sections associated with the sequence of line bundles, thereby extending the earlier results of \cite{SZ99, Bay16} on mass equidistribution. Although our primary focus is on mass distributions, we note that the zero distributions for general sequences of line bundles have been extensively studied (cf. \cite{Ba,BCM, BG, CMM, CLMM}).

The structure of the paper is as follows: Section 2 reviews the necessary background on K\"{a}hler geometry, holomorphic line bundles, and Bergman kernel asymptotics. Section 3 introduces the definitions and fundamental properties of mass distributions, including the proofs of almost sure equidistribution of masses. Section 4 presents the main result, focusing on the proof of asymptotic normality for mass measures.

\section{Preliminaries} \label{Sec2}

Let $(X, J, \omega)$ be a compact K\"{a}hler manifold of complex dimension $n$, where $\omega$ is the K\"{a}hler form and $J$ is the complex structure compatible with $\omega$. Let $\{(L_p, h_p)\}_{p \geq 1}$ be a sequence of holomorphic line bundles on $X$, each equipped with a Hermitian metric $h_p$ of class $\mathscr{C}^3$. For each line bundle, the metric $h_p$ induces a Chern connection $\nabla_{L_p}$, and the associated curvature $R_{L_p} = \nabla_{L_p}^2$ defines the Chern curvature form of $L_p$,
\begin{equation*}
c_1(L_p, h_p) = \frac{i}{2\pi} R_{L_p},
\end{equation*}
which represents the first Chern class of $L_p$ in the cohomology group $H^2(X, \mathbb{Z})$. Locally, the curvature can be written as
\begin{equation*}
c_1(L_p, h_p) = dd^c \varphi_p,
\end{equation*}
where $dd^c = \frac{i}{\pi} \partial \overline{\partial}$ is the complex differential operator, and $\varphi_p$ is the local weight function associated with the Hermitian metric $h_p$. Specifically, $\varphi_p$ is a real-valued $\mathscr{C}^3$ function, defined locally in terms of a holomorphic frame $e_p$ of $L_p$ by the relation:
\begin{equation*}
h_p(e_p, e_p) = |e_p|_{h_p}^2 = e^{-2\varphi_p}.
\end{equation*}

To facilitate local analysis across the sequence of line bundles, we work on contractible Stein open subsets of $X$, such as coordinate balls or polydiscs. If $\mathcal{U} \subset X$ is a contractible Stein open subset  then by Cartan’s theorem B (see e.g., \cite{HOR90}, p.201), $H^1(\mathcal{U}, \mathcal{O}^*) \cong H^2(\mathcal{U}, \mathbb{Z})$, where $H^1(\mathcal{U}, \mathcal{O}^*)$ is the first cohomology group with coefficients in the sheaf of invertible holomorphic functions. Since $\mathcal{U}$ is contractible, we have $H^2(\mathcal{U}, \mathbb{Z}) = 0$, hence $H^1(\mathcal{U}, \mathcal{O}^*) = 0$. By a classical result of Oka \cite{Oka39}, this implies that any holomorphic line bundle over $\mathcal{U}$ is trivial. Consequently, holomorphic frames exist for every line bundle on such subsets, allowing for explicit computation of weights and curvature forms.

A line bundle $(L_p, h_p)$ is called \emph{positive} or \emph{semi-positive} if its curvature form satisfies $c_1(L_p, h_p)(v, \overline{v}) > 0$ or $c_1(L_p, h_p)(v, \overline{v}) \geq 0$, respectively, for all nonzero tangent vectors $v \in T^{1,0}_x X$, $x \in X$. Equivalently, this means that the weight function $\varphi_p$ is strictly plurisubharmonic or plurisubharmonic, respectively, i.e., $dd^c \varphi_p > 0$ or $dd^c \varphi_p \geq 0$ pointwise.

In our setting, the sequence of curvature forms satisfies the Diophantine condition \eqref{appr}, which implies the positivity of the line bundles for large $p$. Consequently, by Grauert’s result (cf. Proposition 2.4 in \cite{CMM}), these line bundles are ample, and $X$ is a projective manifold.

To measure the distance between any two points $x, y \in X$, we utilize the Riemannian distance induced by the K\"{a}hler structure. Specifically, the K\"{a}hler form $\omega$ and the complex structure $J$ on $X$ define a Riemannian metric $g$ on $X$ via
\begin{equation*}
g^{TX} (u, v):= \omega(u, Jv) \quad \text{for all } u, v \in TX.
\end{equation*}
Given a piecewise smooth curve $\gamma: [a, b] \to X$ such that $\gamma(a) = x$ and $\gamma(b) = y$, the length $\mathscr{L}(\gamma)$ of the curve $\gamma$ is given by
\begin{equation*}
\mathscr{L}(\gamma) = \int_{a}^{b} \sqrt{g_{\gamma(t)}(\dot{\gamma}(t), \dot{\gamma}(t))} \, dt,
\end{equation*}
and the Riemannian distance $d(x, y)$ between $x$ and $y$ is defined as
\begin{equation*}
d(x, y) = \inf\left\{ \mathscr{L}(\gamma) : \gamma(a) = x, \, \gamma(b) = y \right\}.
\end{equation*}

The Riemannian volume form on $(X, g^{TX})$ will be denoted by $dV_X=\omega^n / n!,$ and when it is clear from the context we drop the subscript indicating $X$.

\subsection{Reference Covers and Special Metric Structures} \label{metricss}

Let $(U, z)$, with $z = (z_{1}, \ldots, z_{n})$, denote local coordinates centered at a point $x \in X$. The closed polydisk around a point $y \in U$ with equilateral radius $(r, \ldots, r)$ (with $r > 0$) is defined by
\begin{equation*}
P^{n}(y, r) := \{ z \in U : |z_{j} - y_{j}| \leq r,\quad j = 1, 2, \ldots, n \}.
\end{equation*}
We say that the coordinates \((U,z)\) are \emph{Kähler at \(y\)} if the Kähler form has the expansion
\begin{equation}\label{kahcor}
    \omega_z
\;=\;
\frac{i}{2}\sum_{j=1}^{n}dz_{j}\wedge d\bar z_{j}
\;+\; O\bigl(\|z-y\|^{2}\bigr)
\quad\text{on }U.
\end{equation}

\begin{defn} \label{refed}
A \emph{reference cover} of $X$ is defined as follows. For $j = 1, 2, \ldots, N$, consider a set of points $x_{j} \in X$ and:
\begin{itemize}
  \item[(a)] Stein, open, simply connected coordinate neighborhoods $(U_{j}, w^{(j)})$ centered at $x_{j} \equiv 0$.
  \item[(b)] Radii $R_{j} > 0$ such that $P^{n}(x_{j}, 2R_{j}) \Subset U_{j}$ and for every $y \in P^{n}(x_{j}, 2R_{j})$ there exist coordinates on $U_{j}$ that are K\"{a}hler at $y$.
  \item[(c)] The covering $X = \bigcup_{j=1}^{N} P^{n}(x_{j}, R_{j})$.
\end{itemize}
Once a reference cover is chosen, we set $R = \min\{ R_{j} : 1 \leq j \leq N \}.$
\end{defn}

The construction of a reference cover is straightforward. For any $x \in X$, choose a Stein, open, simply connected neighborhood $U$ of $0 \in \mathbb{C}^{n}$ (for example, a round ball in $\mathbb{C}^{n}$) so that, under a suitable chart, $x \equiv 0$. Select $R > 0$ such that $P^{n}(x, R) \Subset U$ and for every $y \in P^{n}(x, R)$ there exist coordinates on $U$ that are K\"{a}hler at $y$. By the compactness of $X$, one may then choose a finite collection $\{ x_{j} \}_{j=1}^{N}$ satisfying these conditions.

Next, consider the differential operators $D_{w}^{\alpha}$ (with $\alpha \in \mathbb{N}^{2n}$) defined on $U_{j}$ with respect to the real coordinates corresponding to $w = w^{j}$. For $\varphi \in \mathscr{C}^{k}(U_{j})$, define
\begin{equation*}
    \|\varphi\|_{k} = \|\varphi\|_{k,w} = \sup \{ |D_{w}^{\alpha} \varphi(w)| : w \in P^{n}(x_{j}, 2R_{j}),\; |\alpha| \leq k \}.
\end{equation*}
Let $(L, h)$ be a Hermitian holomorphic line bundle on $X$ with smooth metric $h$. For $k \leq l$, define
\begin{equation*}
\|h\|_{k, U_{j}} := \inf \{ \|\varphi_{j}\|_{k} : \varphi_{j} \in \mathscr{C}^{l}(U_{j}) \text{ is a weight of } h \text{ on } U_{j} \},
\end{equation*}
and set    
\begin{equation*}
\|h\|_{k} = \max \{ 1, \; \|h\|_{k, U_{j}} : 1 \leq j \leq N \}.
\end{equation*}
Here, $\varphi_{j}$ is called a weight of $h$ on $U_{j}$ if there exists a holomorphic frame $e_{j}$ of $L$ on $U_{j}$ such that
\begin{equation*}
|e_{j}|_{h} = e^{-\varphi_{j}}.
\end{equation*}

We now state the following key lemma which was originally stated
and proved in \cite{CMM}. While a version of the result was reproduced in \cite{BG1} adapted from \cite{CMM}.

\begin{lem}[Reference Cover Lemma] \label{refe}
Let a reference cover of $X$ be given. Then there exists a constant $C > 0$, depending on the reference cover, such that the following holds. For any Hermitian line bundle $(L, h)$ on $X$, any $j \in \{1, \ldots, N\}$, and any $x \in P^{n}(x_{j}, R_{j})$, there exist coordinates $z = (z_{1}, \ldots, z_{n})$ on $P^{n}(x, R)$, centered at $x \equiv 0$, which are K\"{a}hler coordinates for $x$ and for which the Hermitian line bundle $(L, h)$ admits a weight $\varphi$ on $P^{n}(x, R)$ of the form
\begin{equation*}
\varphi(z) = \mathrm{Re}(\psi(z)) + \sum_{j=1}^{n} \lambda_{j} |z_{j}|^{2} + \widetilde{\varphi}(z),
\end{equation*}
where $\psi$ is a holomorphic polynomial of degree at most $2$, $\lambda_{j} \in \mathbb{R}$, and
\begin{equation*}
|\widetilde{\varphi}(z)| \leq C\,\|h\|_{3}\,\|z\|^{3} \quad \text{for } z \in P^{n}(x, R).
\end{equation*}
\end{lem}

Note that at the point $x \equiv 0$ (where K\"{a}hler coordinates are available as in \eqref{kahcor}), the K\"{a}hler form simplifies to
\begin{equation*}
\omega_{x} = i \sum_{j=1}^{n} \frac{1}{2} dz_{j} \wedge d\overline{z_{j}}.
\end{equation*}
Moreover, using the local representation of $c_1(L_p, h_p)$ together with Lemma~\ref{refe}, we have
\begin{equation*}
c_1(L_p, h_p)_x = dd^c \varphi_p(0) = i \sum_{j=1}^{n} \frac{\lambda_j^p}{\pi} dz_{j} \wedge d\overline{z_{j}}.
\end{equation*}
The Diophantine approximation \eqref{appr} then implies
\begin{equation}
\lim_{p \to \infty} \frac{\lambda_j^p}{\pi A_p} = \frac{1}{2}, \quad \text{for } j = 1, 2, \ldots, n,
\end{equation}
which in turn yields
\begin{equation} \label{lamdas}
\lim_{p \to \infty} \frac{\lambda_1^p \cdots \lambda_n^p}{A_p^n} = \left( \frac{\pi}{2} \right)^n.
\end{equation}

\subsection{Bergman Kernel Asymptotics} \label{Bergman}

Let $L^{2}(X, L_{p})$ denote the completion of $\mathscr{C}^{\infty}(X, L_{p})$ with respect to the norm defined in (\ref{innerprod}). Denote by
\begin{equation*}
\Pi_{p}: L^{2}(X, L_{p}) \longrightarrow H^{0}(X, L_{p})
\end{equation*}
the orthogonal projection onto the space of global holomorphic sections. The Bergman kernel $K_{p}(x, y)$ is defined as the integral kernel of this projection.
If $\{S^{p}_{j}\}_{j=1}^{d_{p}}$ is an orthonormal basis for $H^{0}(X, L_{p})$, then we can express the Bergman kernel in terms of this basis as follows
\begin{equation}\label{offds}
K_{p}(x, y) = \sum_{j=1}^{d_{p}} S^{p}_{j}(x) \otimes S^{p}_{j}(y)^{*} \in L_{p, x} \otimes L_{p, y}^{*},
\end{equation}
where $S^{p}_{j}(y)^{*} = \langle \, \cdot \,, S^{p}_{j}(y) \rangle_{h_{p}} \in L^{*}_{p, y}.$
The restriction of the Bergman kernel to the diagonal is called the Bergman kernel function of $H^{0}(X, L_{p})$ and is denoted by
\begin{equation}\label{diags}
B_{p}(x) := K_{p}(x, x) = \sum_{j=1}^{d_{p}} |S^{p}_{j}(x)|_{h_{p}}^{2}.
\end{equation}
This function  enjoys the dimensional density property
\begin{equation*}
\dim H^{0}(X, L_{p}) = \int_{X} B_{p}(x) \, dV(x) = A_{p}^{n}\,\mathrm{Vol}(X) + o(A_{p}^{n}).
\end{equation*}
Moreover, it satisfies the following variational principle
\begin{equation}\label{variat}
B_{p}(x) = \max \Big\{ |S(x)|^{2}_{h_{p}} : S \in H^{0}(X, L_{p}), \ \|S\|_{p} = 1 \Big\},
\end{equation}
which holds for every $x \in X$ at which the local weight $\varphi_{p}(x) > -\infty$ (with $\varphi_{p}$ defined as above for the metric $h_{p}$).

We also define the \emph{normalized Bergman kernel}
\begin{equation}
N_{p}(x, y) := \frac{|K_{p}(x, y)|_{h_{p, x} \otimes h_{p, y}^{*}}}{\sqrt{B_{p}(x)}\,\sqrt{B_{p}(y)}}, \quad x,y \in X,
\end{equation}
which will play a central role in the analysis throughout this work.

The asymptotic results on the Bergman kernel will be needed in the sequel. We begin by stating the diagonal asymptotics of the Bergman kernel function. The following result is an immediate corollary of \cite[Theorem 1.3]{CMM} which holds under considerably more general assumptions.

\begin{thm}\label{bkdiag}
Let $(X, \omega)$ be a compact K\"{a}hler manifold of complex dimension $n$, and let $\{(L_p, h_p)\}_{p \geq 1}$ be a sequence of holomorphic line bundles endowed with Hermitian metrics $h_p$ of class~$\mathscr{C}^3$, satisfying the Diophantine approximation condition. Suppose that $\|h_p\|_3^{1/3} / \sqrt{A_p} \to 0$ as $p \to \infty$. Then the Bergman kernel functions satisfy
\[
B_p(x) = A_p^n(1 + o(1)) \quad \text{uniformly on } X \text{ as } p \to \infty.
\]
\end{thm}

The next result  shows an exponential off-diagonal decay of the Bergman kernels $K_p (x,y)$. This result was obtained in \cite[Theorem 1.4]{BCM} in a more general geometric setting:

\begin{thm}\label{bkoff}
Under the same assumptions as in Theorem~\ref{bkdiag}, there exist constants \( C_1, C_2 > 0 \) and an integer \( p_0 \ge 1 \) such that for all \( x, y \in X \) and \( p > p_0 \), the Bergman kernels satisfy
\[
|K_p(x, y)|_{h_p}^2 \le C_1\, A_p^{2n} \exp\big(-C_2 \sqrt{A_p}\, d(x, y)\big).
\]
\end{thm}

\vspace{0.6 em}

In order to study the near-diagonal asymptotics of the Bergman kernel, we introduce the following convention. Let $V \subset X$ and $U \subset \mathbb{C}^n$ be open subsets, with $x_0 \in V$ and $0 \in U$. Consider a K\"{a}hler coordinate chart $\Psi: V \to U$ such that $\Psi(x_0) = 0$, providing a local identification of $V$ with $U$.
For $z, w \in U \subset \mathbb{C}^n$, we express their preimages under $\Psi$ as:
\begin{equation*}
\Psi^{-1}(z) = x_0 + z, \quad \Psi^{-1}(w) = x_0 + w,
\end{equation*}
with the Bergman kernel in these coordinates written as:
\begin{equation*}
K_p(\Psi^{-1}(z), \Psi^{-1}(w)) := K_p(x_0 + z, x_0 + w).
\end{equation*}

Here, $x_0 + z$ corresponds to the point in $X$ obtained by displacing $x_0$ by $z$ in the local coordinates. While differences such as $\Psi^{-1}(z) - x_0$ are not globally meaningful on $X$, the displacement $z - 0$ is well-defined in the vector space $\mathbb{C}^n$, where we treat the displacement as a well-posed vector operation. This \emph{linearization} simplifies the analysis of the Bergman kernel near the diagonal by allowing us to focus on displacements in the local coordinate system. This approach follows standard conventions used in the literature (e.g., \cite{SZ08}, \cite{SZ10}) and is crucial for deriving the asymptotic behavior of the Bergman kernel near the diagonal, as formalized in the theorems (cf. \cite[Theorem 4.3]{BG1}  and \cite[Theorem 2.3]{Bay16}) ]that follow.

\begin{thm}
\label{nearbk}
Assume the hypotheses of Theorem~\ref{bkdiag}. Fix a point \(x\in X\) and choose K\"ahler coordinates centred at \(x\) provided by Lemma~\ref{refe} (Reference-Cover Lemma).  
Then
\begin{equation*}
  \frac{\Bigl|K_p\Bigl(x+\frac{u}{\sqrt{A_p}},\,x+\frac{\overline{v}}{\sqrt{A_p}}\Bigr)\Bigr|^2\,
        \exp\!\Bigl[-2\,\varphi_p\Bigl(x+\frac{u}{\sqrt{A_p}}\Bigr)\Bigr]
        \,\exp\!\Bigl[-2\,\varphi_p\Bigl(x+\frac{\overline{v}}{\sqrt{A_p}}\Bigr)\Bigr]}
       {A_p^{2n}\,\exp\!\Bigl[-2\sum_{j=1}^n \frac{\lambda_j^p}{A_p}\,\lvert u_j-\overline{v}_j\rvert^2\Bigr]}
  \;\longrightarrow\; 1,
\end{equation*}
locally uniformly on $\mathbb{C}_{u}^n \times \mathbb{C}_{v}^n$ as $p\to\infty$.
\end{thm}

As an immediate consequence of Theorems \ref{bkdiag} and \ref{nearbk}, we obtain 
the following near-diagonal asymptotics for the normalized Bergman kernel.
\begin{thm}\label{normbk}
Under the same assumptions as in Theorem~\ref{nearbk}, the normalized Bergman kernel satisfies the following asymptotics
\begin{equation*}
  N_p\Bigl(x + \frac{u}{\sqrt{A_p}},\, x + \frac{\overline{v}}{\sqrt{A_p}}\Bigr)
  \;=\;
  \exp\Bigl(-\sum_{j=1}^n \frac{\lambda_j^p}{A_p}\,\lvert u_j-\overline{v}_j\rvert^2\Bigr)
  \Bigl(1+o(1)\Bigr),
\end{equation*}
locally uniformly on $\mathbb{C}_u^n \times \mathbb{C}_v^n$ as $p \to \infty$.
\end{thm}

\section{Mass Equidistribution for Random Holomorphic Sections} \label{Section3}

In this section, we establish the equidistribution of masses for random holomorphic sections associated with the given geometric data. In other words, we prove that for almost every sequence of random holomorphic sections, the normalized mass measures converge to the canonical volume measure on $X$, which is a manifestation of quantum ergodicity.

Let $\{S_j^p\}_{j=1}^{d_p}$ be an orthonormal basis (ONB) for $H^0(X, L_p)$ with respect to the inner product defined in \eqref{innerprod}. A Gaussian random holomorphic section is given by
\begin{equation}
    s_p = \sum_{j=1}^{d_p} a_j^p S_j^p,
\end{equation}
where $\{a_j^p\}_{j=1}^{d_p}$ are independent and identically distributed (i.i.d.) standard complex Gaussian random variables. That is,  the real and imaginary parts of each coefficient, $\operatorname{Re}(a_j^p)$ and $\operatorname{Im}(a_j^p)$, are independent real Gaussian random variables with mean zero and variance $1/2$. Equivalently, the coefficients satisfy 
\begin{equation*}
\mathbf{E}[a_j^p] = 0, \quad \mathbf{E}[a_j^p \overline{a_k^p}] = \delta_{jk}, \quad \text{and} \quad \mathbf{E}[a_j^p a_k^p] = 0.
\end{equation*}
This construction induces a natural $d_p$-fold probability measure $\mathbb{P}_p$ on $H^0(X, L_p)$ via the identification  
\begin{equation*}
H^0(X, L_p) \simeq \mathbb{C}^{d_p}, \quad s_p \mapsto (a_1^p, \dots, a_{d_p}^p),
\end{equation*}
under which $\mathbb{P}_p$ corresponds to the standard Gaussian measure on $\mathbb{C}^{d_p}$. Since this measure is invariant under unitary transformations, $\mathbb{P}_p$ does not depend on the choice of ONB and thus defines a canonical measure on $H^0 (X,L_p).$

We begin by analyzing the expected distribution of masses for random holomorphic sections which provides the first step in understanding their asymptotic almost sure behaviour.

\begin{prop}\label{ExpMass}
    Let $(X, \omega)$ be a compact K\"{a}hler manifold of complex dimension $n$, and let $\{(L_p, h_p)\}_{p \geq 1}$ be a sequence of holomorphic line bundles equipped with Hermitian metrics $h_p$ of class $\mathscr{C}^3$, satisfying the Diophantine approximation condition, and $\|h_p\|_3 ^{1/3} / \sqrt{A_p}\to 0$ as $p \to \infty$. Then 
    \begin{equation}
        \mathbb{E}[\mathcal{M}_p(s_p)] \to dV 
    \end{equation}
    as $p \to \infty$ in the weak-* topology of measures on $X$.
\end{prop}

\begin{proof} Observe that for any $x\in X$ we have
$$|s_p(x)|_{h_p}^2= \sum_{1\leq j,k\leq d_p} a_{j}^p \, \overline{a_{k}^p} \, \big\langle S_j ^p (x), S_k ^p (x) \big\rangle_{h_p}.$$ Since $a_j ^p$ are i.i.d. of mean zero and variance one, we have $\mathbb{E}[|s_p(x)|_{h_p}^2]=B_p (x)$ for every $x\in X$ . Then by Fubini's Theorem, for any $\phi \in \mathscr{C}^0 (X)$, we have 
$$\mathbb{E}\Big[ \int_{X}\phi(x) \mathcal{M}_p (s_p)(x)\Big]=\frac{1}{A_p ^n}\int_{X}\phi(x)B_p (x)dV(x)$$
Hence the assertion follows from Theorem \ref{bkdiag}.
\end{proof}
Let $\phi:X\rightarrow \mathbb{C}$ be a continuous function, then consider the following random variables as defined in the introduction, namely
$$\mathcal{M}_{p}^{\phi}: H^0 (X,L_p) \rightarrow \mathbb{C}$$ defined as 
$$\mathcal{M}_{p}^\phi(s_p):=\frac{1}{A_p^n}\int_{X}|s_p(x)|_{h_p}^2 \phi(x)dV(x)$$
\begin{rem}
    In this section, we allow $\phi$ to be complex-valued. However, in Section 4, we restrict $\phi$ to be real-valued, which is crucial.
\end{rem}

\begin{thm}\label{varmass}
Under the same geometric assumptions as in Proposition \ref{ExpMass}, we have
\begin{equation}
    \mathrm{Var}[\mathcal{M}_{p}^{\phi}(s_p)] = \frac{1}{A_p^n} \int_X |\phi(x)|^2 \, dV(x) + o\left(\frac{1}{A_p^n}\right).
\end{equation}
\end{thm}

\begin{proof}
We begin by writing the second moment of $\mathcal{M}_p^\phi(s_p)$
\begin{align}\label{earliear}
\mathbb{E}\big[|\mathcal{M}_{p}^{\phi}(s_p)|^2\big] &= \frac{1}{A_p^{2n}} \int_{X \times X} \mathbb{E}\big[|s_p(x)|^2_{h_p} |s_p(y)|^2_{h_p}\big] \phi(x) \overline{\phi(y)}  \, dV(x) \, dV(y) \\
&= \frac{1}{A_p^{2n}} \int_{X \times X} B_p(x) B_p(y) \mathbb{E}\left[\frac{|s_p(x)|_{h_p}^2}{B_p(x)} \frac{|s_p(y)|_{h_p}^2}{B_p(y)}\right] \phi(x) \overline{\phi(y)} \, dV(x) \, dV(y).
\end{align}

By Lemma 5.2 of \cite{Zel18}, we have the following 
\begin{equation}
\mathbb{E}\left[\frac{|s_p(x)|_{h_p}^2}{B_p(x)} \frac{|s_p(y)|_{h_p}^2}{B_p(y)}\right] = 1 + N_p(x,y)^2.
\end{equation}
Substituting this into the earlier expression \eqref{earliear}, we obtain
\begin{equation}
\mathbb{E}\big[|\mathcal{M}_{p}^{\phi}(s_p)|^2\big] = \big|\mathbb{E}\big[|\mathcal{M}_{p}^{\phi}(s_p)|\big]\big|^2 + \frac{1}{A_p^{2n}} \int_{X \times X} \phi(x) \overline{\phi(y)}  B_p(x) B_p(y) N_p(x,y)^2 \, dV(x) \, dV(y).
\end{equation}
Hence,
\begin{equation}
\mathrm{Var}[\mathcal{M}_{p}^{\phi}(s_p)] = \frac{1}{A_p^{2n}} \int_{X \times X} B_p(x) B_p(y) N_p(x,y)^2 \phi(x) \overline{\phi(y)}  \, dV(x) \, dV(y).
\end{equation}

Now, to study the asymptotic behavior of the variance, we fix a constant $b>0$, which will later be taken sufficiently large, and divide the integration region into two parts:
\begin{itemize}
    \item \emph{Far-off diagonal region}: $\{ (x,y) \in X \times X : d(x,y) \geq b \frac{\log A_p}{\sqrt{A_p}} \}$,
    \item \emph{Near-diagonal region}: $\{ (x,y) \in X \times X : d(x,y) < b \frac{\log A_p}{\sqrt{A_p}} \}$.
\end{itemize}

In the far-off diagonal region,  Theorems~\ref{bkoff} and \ref{bkdiag} imply that for any $b \geq \frac{n+1}{C_2}$ and  sufficiently large $p$, we have $N_p(x,y) = O(A_p^{-n-1})$. Therefore, the contribution from this region is bounded by $O(A_p^{-n-1})$, which is negligible relative to the main term order $1/A_p^{n}$ in the limit as $p \to \infty$. Thus, the contribution from the far-off diagonal set does not affect the leading-order behavior of the variance.

Now, we consider the near-diagonal region. Let $x \in X$ and choose local coordinates $z = (z_1, \dots, z_n)$ on  $P^n(x, R)$, centered at $x \equiv 0$ and K\"{a}hler at $x$ (provided by Reference Cover Lemma). Within this chart, we make the change of variables $y = x + v / \sqrt{A_p}$ which corresponds to $z(y)=v/\sqrt{A_p}$ in coordinates. Since $z(x)=0$ and the coordinates are K\"{a}hler  at $x$, we have
$$
\omega = \frac{i}{2} \sum_{j=1}^n dz_j \wedge d\bar{z}_j + O(\|z\|^2),
$$
Substituting  $y = x + v/\sqrt{A_p}$, we get
\begin{equation}
    \omega(x + v/\sqrt{A_p}) = \frac{i}{2A_p} \sum_{j=1}^n dv_j \wedge d\bar{v}_j + O\left(\frac{\|v\|^2}{A_p^2}\right).
\end{equation}
Thus, the volume form becomes
\begin{equation}
    dV\left(x + \frac{v}{\sqrt{A_p}}\right) =\frac{1}{n!}\Bigg[\frac{i}{2A_p}\sum_{j=1}^{n}dv_j \wedge d \bar{v}_j + O\left(\frac{\|v\|^2}{A_p ^2}\right)\Bigg]^n= \frac{1}{A_p^n} \left[ dV_E(v) + O\left(\frac{(\log A_p)^2}{A_p}\right) \right],
\end{equation}
for $\|v\|\leq b\log A_p$, where $dV_E(v) = \prod_{j=1}^n \frac{i}{2} dv_j \wedge d\bar{v}_j$ is the Euclidean volume form on $\mathbb{C}^n$.

Since $\phi \in \mathscr{C}^0(X)$ is uniformly continous on $X$, we have $\phi(x + v/\sqrt{A_p}) = \phi(x) + o(1)$ uniformly for $x\in X$ and $\|v\|\leq b\log A_p$ as $p\rightarrow \infty$. Also, by Theorems \ref{bkdiag} and \ref{normbk},  we obtain the following expression for the variance

\begin{align*}
\mathrm{Var}[\mathcal{M}_{p}^{\phi}(s_p)] = 
\frac{1}{A_p^{2n}} &\int_X \phi(x) \left( A_p^n + o(A_p^n) \right) \times \int_{\|v\| \leq b \log A_p} e^{-2 \sum_{j=1}^n \frac{\lambda_j^p}{A_p} |v_j|^2} \left(1 + o(1)\right) \left( A_p^n + o(A_p^n) \right) \\
&  \quad\times \left(\overline{\phi(x)}  + o(1)\right) \frac{1}{A_p^n} \left[ dV_E(v) + O\left(\frac{(\log A_p)^2}{A_p}\right) \right] \, dV(x).
\end{align*}

Simplifying, we get
\begin{equation}
\mathrm{Var}[\mathcal{M}_{p}^{\phi}(s_p)] = \frac{1}{A_p^n} \int_X |\phi(x)|^2 \, dV(x) \int_{\|v\| \leq b \log A_p} e^{-2 \sum_{j=1}^n \frac{\lambda_j^p}{A_p} |v_j|^2} \, dV_E(v) + o\left(\frac{1}{A_p^n}\right).
\end{equation}

Next, we evaluate the asymptotics of the inner integral. Note that since $\lambda_j ^p  /A_p\rightarrow \pi/2$ as $p\rightarrow \infty$, for sufficiently large $b$, we have
\begin{equation}
    \int_{\|v\| \geq b \log A_p} e^{-2 \sum_{j=1}^n \frac{\lambda_j^p}{A_p} |v_j|^2} \, dV_E(v) = O(A_p^{-n-1}).
\end{equation}
Therefore, we can extend the integral over  $\|v\| \geq b \log A_p$ to the entire space $\mathbb{C}^n$. Furthermore, since
\begin{equation}
    \int_{\mathbb{C}^n} e^{-2 \sum_{j=1}^n \frac{\lambda_j^p}{A_p} |v_j|^2} \, dV_E(v) = \frac{\pi^n}{2^n} \frac{A_p^n}{\lambda_1^p \cdots \lambda_n^p} = 1 + o(1),
\end{equation}
we obtain the desired first-order asymptotics for variance.
\end{proof}

In order to prove an almost sure result, we consider the product probability space
\begin{equation*}
(\mathscr{H}, \mathbb{P}) = \prod_{p=1}^\infty \big( H^0(X, L_p), \mathbb{P}_p \big),
\end{equation*}
whose elements are sequences $(s_1, s_2, \ldots)$ of random holomorphic sections.

\begin{thm}
    Let $(X, \omega)$ be a compact K\"{a}hler manifold of complex dimension $n$, and let $\{(L_p, h_p)\}_{p \geq 1}$ be a sequence of holomorphic line bundles equipped with Hermitian metrics $h_p$ of class $\mathscr{C}^3$, satisfying the Diophantine approximation condition, and $\|h_p\|_3 ^{1/3} / \sqrt{A_p}\to 0$ as $p \to \infty$.  If 
    \begin{equation*}
    \sum_{p=1}^\infty \frac{1}{A_p^n} < \infty,
    \end{equation*}
    then for $\mathbb{P}$-almost every sequence of random holomorphic sections, we have
    \begin{equation}
        \mathcal{M}_p(s_p) \to dV
    \end{equation}
    as $p \to \infty$ in the weak-* topology of measures on $X$.
\end{thm}

\begin{proof}
    Consider a random sequence $\mathbf{s} = (s_1, s_2, \ldots)$ in $\mathscr{H}$, and define the random variables
    \begin{equation}
    M_p(\mathbf{s}) := \mathcal{M}_p^\phi(s_p),
    \end{equation}
    where $\phi \in \mathscr{C}^0(X)$ is a test function. Given $\epsilon > 0$, by Proposition \ref{ExpMass}, we know that $\mathbb{E}[\mathcal{M}_p(s_p)] \to dV$ in the weak-* topology. Thus, for sufficiently large $p$,
    \begin{equation}
        \big|\mathbb{E}[\mathcal{M}_p^\phi(s_p)] - \int_X \phi(x) dV(x)\big| \leq \epsilon.
    \end{equation}

   Now, we estimate the probability
    \begin{equation}
    \mathbb{P}\big(|M_p(\mathbf{s}) - \int_X \phi(x) dV(x)| > 2\epsilon\big) = \mathbb{P}_p\big(|\mathcal{M}_p^\phi(s_p) - \int_X \phi(x) dV(x)| > 2\epsilon\big).
    \end{equation}
    Applying the triangle inequality,
    \begin{equation}
    |\mathcal{M}_p^\phi(s_p) - \int_X \phi(x) dV(x)| \leq |\mathcal{M}_p^\phi(s_p) - \mathbb{E}[\mathcal{M}_p^\phi(s_p)]| + |\mathbb{E}[\mathcal{M}_p^\phi(s_p)] - \int_X \phi(x) dV(x)|.
    \end{equation}
    Since the second term is bounded by $\epsilon$ for large $p$, we have
    \begin{equation}
    \mathbb{P}_p\big(|\mathcal{M}_p^\phi(s_p) - \int_X \phi(x) dV(x)| > 2\epsilon\big) \leq \mathbb{P}_p\big(|\mathcal{M}_p^\phi(s_p) - \mathbb{E}[\mathcal{M}_p^\phi(s_p)]| > \epsilon\big).
    \end{equation}
    By Theorem \ref{varmass} and Markov's inequality,
    \begin{equation}
    \mathbb{P}_p\big(|\mathcal{M}_p^\phi(s_p) - \mathbb{E}[\mathcal{M}_p^\phi(s_p)]| > \epsilon\big) \leq \frac{\mathrm{Var}[\mathcal{M}_p^\phi(s_p)]}{\epsilon^2} = \frac{1}{\epsilon^2 A_p^n} \bigg(\int_X |\phi(x)|^2 dV(x) + o(1)\bigg).
    \end{equation}

    Since $\sum_{p=1}^\infty A_p^{-n} < \infty$, the Borel-Cantelli lemma implies that for $\mathbb{P}$-almost every sequence $\mathbf{s} \in \mathscr{H}$, there exists $p_1(\mathbf{s})$ such that for all $p \geq p_1(\mathbf{s})$,
    \begin{equation}
    \big|M_p(\mathbf{s}) - \int_X \phi(x) dV(x)\big| \leq 2\epsilon.
    \end{equation}

    Finally, as $\epsilon > 0$ is arbitrary, we conclude that
    \begin{equation*}
    \mathcal{M}_p(s_p) \to dV
    \end{equation*}
    in the weak-* topology of measures on $X$, completing the proof.
\end{proof}
\begin{rem}
    This theorem, under a summability condition, generalizes previous results in the literature, which are primarily focused on tensor powers of a prequantum line bundle. In particular, in our setting, if $ (L_p, h_p) = (L^{\otimes p}, h^{\otimes p}) $, where $ L $ is a prequantum line bundle over a compact K\"{a}hler manifold $ X $ of dimension $ n > 1 $, then $ A_p = p $, and the summability condition is automatically satisfied. (Note that $
\|h_p\|_{3} \leq D_h p,
$ where $D_h$ is the constant depending on $h$ and the reference cover. Hence $\|h_p\|_{3}^{1/3} A_p^{-1/2} \leq D_h ^{1/3} p^{-1/6}\rightarrow 0$ as $p\rightarrow \infty$ 
).\end{rem}

\section{Asymptotic Normality for Masses of Random Sections}

The investigation of asymptotic normality in the context of linear statistics of mass distributions of random holomorphic sections, presents a fascinating and unexplored challenge. This is in contrast to the well-studied case of smooth linear statistics of the zero divisors of random holomorphic sections, where asymptotic normality has been established under various settings.

In this direction, the seminal work of Sodin and Tsirelson \cite{ST} is particularly noteworthy. They established an asymptotic normality result for certain nonlinear functionals associated with Gaussian processes in the complex plane, serving as a foundational tool in our subsequent analysis.

\subsection{Central Limit Theorem for Nonlinear Functionals of Complex Gaussian Processes}

Let $(\Omega, \mu)$ be a measure space with a finite positive measure $\mu$, and consider a sequence $\{\psi_j\}_{j=1}^\infty$ of complex-valued measurable functions on $\Omega$ satisfying
\begin{equation}
\sum_{j=1}^\infty |\psi_j(x)|^2 = 1, \quad \forall x \in \Omega.
\end{equation}

A \textit{normalized complex Gaussian process} is defined to be a complex-valued random function $G(x)$ on the measure space $(\Omega, \mu)$ of the form
\begin{equation}\label{eq:gaussian_process}
G(x) = \sum_{j=1}^\infty c_j \psi_j(x),
\end{equation}
where $\{c_j\}_{j=1}^\infty$ is a sequence of i.i.d. centered complex Gaussian random variables with variance 1.

The covariance function $\mathcal{K}: \Omega\times \Omega\rightarrow \mathbb{C}$ of $G(x)$ is then given by
\begin{equation}\label{cov}
\mathcal{K}(x, y) := \mathbb{E}[G(x) \overline{G(y)}] = \sum_{j=1}^\infty \psi_j(x) \overline{\psi_j(y)}.
\end{equation}
Observe that $|\mathcal{K}(x, y)| \leq 1$ and $\mathcal{K}(x, x) = 1$ for all $x, y \in \Omega$.

Now, consider a sequence $\{G_p(x)\}_{p\in \mathbb{N}}$ of normalized complex Gaussian processes defined on a finite measure space $(\Omega, \mu)$. Let $f : \mathbb{R}^+ \to \mathbb{R}$ be a measurable function such that
\begin{equation}
\int_0^\infty f(r)^2 e^{-r^2/2} r \, dr < \infty.
\end{equation}
For a bounded, measurable function $\phi : \Omega \to \mathbb{R}$, define the nonlinear functionals
\begin{equation}\label{nonlin}
\mathcal{I}_p^\phi(G_p) = \int_\Omega f(|G_p(x)|) \phi(x) \, d\mu(x).
\end{equation}

The following theorem, due to Sodin and Tsirelson \cite[Theorem 2.2]{ST}, establishes the asymptotic normality of $\ \mathcal{I}_p^\phi(G_p)$.

\begin{thm}\label{sots}
Let $\mathcal{K}_p(x, y)$ denote the covariance function of $G_p(x)$. Assume the following conditions hold for all $\alpha \in \mathbb{N}$:
\begin{itemize}
    \item[(i)] \begin{equation*}
    \liminf_{p \to \infty} \frac{\int_\Omega \int_\Omega |\mathcal{K}_p(x, y)|^{2\alpha} \phi(x) \phi(y) \, d\mu(x) \, d\mu(y)}{\sup_{x \in \Omega} \int_\Omega |\mathcal{K}_p(x, y)| \, d\mu(y)} > 0.
    \end{equation*}
    \item[(ii)] \begin{equation*}
    \lim_{p \to \infty} \sup_{x \in \Omega} \int_\Omega |\mathcal{K}_p(x, y)| \, d\mu(y) = 0.
    \end{equation*}
\end{itemize}

Then the distributions of normalized random variables
\begin{equation*}
\frac{\mathcal{I}_p^\phi(G_p) - \mathbb{E}[\mathcal{I}_p^\phi(G_p)]}{\sqrt{\mathrm{Var}[\mathcal{I}_p^\phi(G_p)]}}
\end{equation*}
converge weakly to the (real) standard normal distribution $\mathcal{N}_{\mathbb{R}}(0, 1)$ as $p \to \infty$.
Moreover, if $f$ is an increasing function, then condition $(i)$ only needs to hold for $\alpha = 1$.
\end{thm}

Their proof combines the classical method of moments with the diagram technique, which enables the computation of moments of nonlinear functionals and their comparison with those of the standard Gaussian distribution. We also remark that condition (ii) ensures the asymptotic vanishing of the variance, i.e., $\mathrm{Var}[\mathcal{I}_{p}^{\phi}(G_{p})] \to 0$ as $p \to \infty$.

\vspace{0.9 em}

In order to prove asymptotic normality for $\mathcal{M}_p^{\phi}(s_p)$, we begin by decomposing it as
\begin{equation}\label{Relation}
    \mathcal{M}_p^{\phi}(s_p) = \mathcal{F}_p^{\phi}(s_p) + \mathcal{R}_p^{\phi}(s_p),
\end{equation}
where
\begin{equation}
    \mathcal{F}_p^{\phi}(s_p) := \int_{X} \frac{|s_p(x)|_{h_p}^2}{B_p(x)} \phi(x) dV(x),
\end{equation}
and
\begin{equation}
    \mathcal{R}_p^{\phi}(s_p) := \int_{X} |s_p(x)|_{h_p}^2 \left(\frac{1}{A_p^n} - \frac{1}{B_p(x)}\right) \phi(x) dV (x).
\end{equation}

Using the asymptotic relation $B_p(x) = A_p^n(1 + o(1))$, it follows that
\begin{equation}\label{maincont}
    \mathrm{Var}[\mathcal{F}_p^{\phi}(s_p)] \sim \mathrm{Var}[\mathcal{M}_p^{\phi}(s_p)],
\end{equation}
i.e., $\mathrm{Var}[\mathcal{F}_p^{\phi}(s_p)] / \mathrm{Var}[\mathcal{M}_p^{\phi}(s_p)] \to 1$ as $p \to \infty$.

Furthermore, since $1/B_p(x) = 1/A_p^n + o(1/A_p^n)$, and following similar steps as in the proof of Theorem \ref{varmass}, we obtain
\begin{equation}\label{remainder}
    \mathrm{Var}[\mathcal{R}_p^{\phi}(s_p)] = o\left(\frac{1}{A_p^n}\right).
\end{equation}

By employing the result of Sodin and Tsirelson, we now establish the following asymptotic normality result for the random variables $\mathcal{F}_p^{\phi}(s_p)$.

\begin{prop}\label{prop}
Under the same assumptions as in Theorem \ref{maint}, the distributions of the random variables
\begin{equation}\label{clte}
    \frac{\mathcal{F}_p^{\phi}(s_p) - \mathbb{E}[\mathcal{F}_p^{\phi}(s_p)]}{\sqrt{\mathrm{Var}[\mathcal{F}_p^{\phi}(s_p)]}}
\end{equation}
converge weakly to the standard Gaussian distribution $\mathcal{N}_{\mathbb{R}}(0, 1)$ as $p \to \infty$.
\end{prop}

\begin{proof}
We begin by adapting the standard description of Gaussian processes in the complex plane to our setting. For each $p\in\mathbb{N}$, we construct normalized Gaussian processes $G_p$ on $X$ as follows. First, choose a measurable section  $e_{L_p}: X\to L_p,$ satisfying $|e_{L_p}(x)|_{h_p}=1$ for every $x\in X$. Next, select an orthonormal basis $\{S_j^p\}_{j=1}^{d_p}$ of $H^0(X,L_p)$, and write each basis element as $S_j^p = \sigma_j^p\, e_{L_p}.$
Define
\begin{equation}\label{pro_new}
F_j^p(x)=\frac{\sigma_j^p(x)}{\sqrt{B_p(x)}}, \quad j=1,\dots,d_p,
\end{equation}
so that, by (\ref{diags}), we have  $\sum_{j=1}^{d_p} |F_j^p(x)|^2 = 1.$ Then a normalized complex Gaussian process on $X$ is given by
\begin{equation}\label{cgp_new}
G_p(x)=\sum_{j=1}^{d_p} a_j^p\, F_j^p(x),
\end{equation}
where the $a_j^p$ are i.i.d. centered complex Gaussian random variables with variance one. Notice that a random holomorphic section 
$s_p=\sum_{j=1}^{d_p} a_j^p\, S_j^p$
can be rewritten as
\begin{equation}\label{sect_new}
s_p(x)=\sqrt{B_p(x)}\, G_p(x)\, e_{L_p}(x),
\end{equation}
which immediately implies
\begin{equation}\label{norm_new}
|G_p(x)|=\frac{|s_p(x)|_{h_p}}{\sqrt{B_p(x)}}.
\end{equation}

We now compute the covariance function of $G_p$. Since the coefficients $a_j^p$ are independent, centered, and satisfy
$\mathbb{E}[|a_j^p|^2]=1,\, \, \mathbb{E}[a_k^p\,\overline{a_l^p}]=0 \, \, \text{for } k\neq l,$
it follows from (\ref{cgp_new}) that
\begin{equation}\label{cov1_new}
\mathcal{K}_p(x,y)=\mathbb{E}\bigl[G_p(x)\,\overline{G_p(y)}\bigr]
=\sum_{j=1}^{d_p} F_j^p(x)\,\overline{F_j^p(y)}.
\end{equation}

Observe that, a standard computation yield
\begin{equation}\label{normb_new}
|K_p(x,y)|_{h_{p,x}\otimes h^*_{p,y}}=\Bigl|\sum_{j=1}^{d_p}\sigma_j^p(x)\,\overline{\sigma_j^p(y)}\Bigr|.
\end{equation}
Combining (\ref{cov1_new}), (\ref{normb_new}) and (\ref{pro_new}), we deduce that
\begin{equation}\label{coveb_new}
N_p(x,y)= |\mathcal{K}_p(x,y)|.
\end{equation}

Next, let $f(r)=r^2$ and identify $(\Omega,\mu)$ with $(X,dV)$. Fix a real-valued continuous function $\phi$ on $X$. Then the nonlinear random functional defined in (\ref{nonlin}) becomes
\begin{equation}\label{samev_new}
\mathcal{I}_p^\phi(G_p)=\int_X|G_p(x)|^2\,\phi(x)\,dV(x)
=\int_X\frac{|s_p(x)|_{h_p}^2}{B_p(x)}\,\phi(x)\,dV(x)
=\mathcal{F}_p^\phi(s_p).
\end{equation}

To verify conditions (i) and (ii) of Theorem \ref{sots}, we split the integration domain into two parts: one where $d(x,y)\ge b\,\frac{\log A_p}{\sqrt{A_p}}$ (the far off-diagonal region) and one where $d(x,y)\le b\,\frac{\log A_p}{\sqrt{A_p}}$ (the near-diagonal region).

First, for the far off-diagonal region, Theorem \ref{bkoff} yields
\begin{equation*}
\lim_{p\to\infty}\sup_{x\in X}\int_{d(x,y)\ge b\,\frac{\log A_p}{\sqrt{A_p}}}N_p(x,y)\,dV(y)
\le \lim_{p\to\infty}\sup_{x\in X}\int_{d(x,y)\ge b\,\frac{\log A_p}{\sqrt{A_p}}}C_1\,e^{-C_2\sqrt{A_p}\,d(x,y)}\,dV(y)=0.
\end{equation*}
For the near-diagonal region, using (\ref{coveb_new}) and the fact that $|\mathcal{K}_p (x,y)|\leq 1$, we obtain
\begin{equation*}
\lim_{p\to\infty}\sup_{x\in X}\int_{d(x,y)\le b\,\frac{\log A_p}{\sqrt{A_p}}}N_p(x,y)\,dV(y)
\le \lim_{p\to\infty}\sup_{x\in X}\int_{d(x,y)\le b\,\frac{\log A_p}{\sqrt{A_p}}}1\,dV(y)=0.
\end{equation*}

Next, we address condition (i). On the far off-diagonal set where 
$d(x,y)\ge b\frac{\log A_p}{\sqrt{A_p}},$
the integrand in the numerator decays at a rate $O(A_p^{-\delta})$, whereas the denominator decays like $O(A_p^{-\delta/2})$ (by Theorem \ref{bkoff}), ensuring that the ratio tends to zero as $p\to\infty$.

Finally, we analyze the near-diagonal region. By linearization on the polydisk $P^n(x,R)\subset U_j$ (with K\"{a}hler coordinates provided by Lemma \ref{refe}) and fixing $x\in X$ in the support of $\phi$ where the supremum in the denominator is attained, for $\alpha\in \Bbb{N}$ we verify that
\begin{equation}\label{st22_new}
\liminf_{p\to\infty}\frac{\displaystyle \int_X\int_{\|v\|\le b\log A_p}N_p^{2\alpha}\Bigl(x,x+\frac{v}{\sqrt{A_p}}\Bigr)\,\phi(x)\,\phi\Bigl(x+\frac{v}{\sqrt{A_p}}\Bigr)\,dv\,dV(x)}
{\displaystyle \int_{\|v\|\le b\log A_p}N_p\Bigl(x,x+\frac{v}{\sqrt{A_p}}\Bigr)\,dv} > 0.
\end{equation}
Employing the near-diagonal asymptotics for $N_p(x,y)$, as given in Theorem \ref{normbk}, this limit becomes
\begin{equation}\label{st22_alt}
\liminf_{p\to\infty}\frac{\displaystyle \int_X\phi(x)\int_{\|v\|\le b\log A_p}\exp\Bigl(-2\alpha\sum_{j=1}^n\frac{\lambda_j^p}{A_p}|v_j|^2\Bigr)(1+o(1))\,\phi\Bigl(x+\frac{v}{\sqrt{A_p}}\Bigr)\,dv\,dV(x)}
{\displaystyle \int_{\|v\|\le b\log A_p}\exp\Bigl(-\sum_{j=1}^n\frac{\lambda_j^p}{A_p}|v_j|^2\Bigr)(1+o(1))\,dv}. 
\end{equation}
Since as $p\to\infty$ we have $\lambda_j^p/A_p\to\pi/2$ and $\phi(x+v/\sqrt{A_p})\to\phi(x)$ uniformly on compact sets in $\mathbb{C}^n$, we obtain
\begin{equation*}
\frac{\displaystyle \int_X \phi^2(x)\,dV(x) \int_{\mathbb{C}^n}\exp\Bigl(-\alpha\pi\sum_{j=1}^n|v_j|^2\Bigr)\,dv}
{\displaystyle \int_{\mathbb{C}^n}\exp\Bigl(-\frac{\pi}{2}\sum_{j=1}^n|v_j|^2\Bigr)\,dv}
=\frac{1}{2^{\alpha n}}\int_X \phi^2(x)\,dV(x) > 0,
\end{equation*}
which completes the verification of (i) and hence the proof.
\end{proof}

\begin{lem} \label{maincontclt}
The random variables
\begin{equation}
    \frac{\mathcal{F}_p^\phi(s_p) - \mathbb{E}[\mathcal{F}_p^\phi(s_p)]}{\sqrt{\mathrm{Var}[\mathcal{M}_p^\phi(s_p)]}}
\end{equation}
converge in distribution to the standard Gaussian random variable $\mathcal{N}_{\mathbb{R}}(0,1)$ as $p \to \infty$.
\end{lem}

\begin{proof}
Observe that we can write
\begin{equation} \label{rewrite}
    \frac{\mathcal{F}_p^\phi(s_p) - \mathbb{E}[\mathcal{F}_p^\phi(s_p)]}{\sqrt{\mathrm{Var}[\mathcal{M}_p^\phi(s_p)]}} = \beta_p \widetilde{\mathcal{F}}_p^\phi(s_p),
\end{equation}
where
$\widetilde{\mathcal{F}}_p^\phi(s_p) = \frac{\mathcal{F}_p^\phi(s_p) - \mathbb{E}[\mathcal{F}_p^\phi(s_p)]}{\sqrt{\mathrm{Var}[\mathcal{F}_p^\phi(s_p)]}} \,\, \text{and}\,\,
\beta_p = \sqrt{\frac{\mathrm{Var}[\mathcal{F}_p^\phi(s_p)]}{\mathrm{Var}[\mathcal{M}_p^\phi(s_p)]}}.
$
From the variance asymptotics in (\ref{maincont}), $\beta_p \to 1$ as $p \to \infty$. By Proposition \ref{prop}, $\widetilde{\mathcal{F}}_p^\phi(s_p)$ converges in distribution to $\mathcal{N}_\mathbb{R} (0,1)$. Denoting the distribution function of $\widetilde{\mathcal{F}}_p^\phi(s_p)$ by $\mathbf{F}_p$, for each $t\in \mathbb{R}$ we have
\begin{equation} \label{cltconv}
    \mathbf{F}_p(t) = \mathbb{P}_p\left(\widetilde{\mathcal{F}}_p^\phi(s_p) \leq t\right) \to \mathbf{\Phi}(t),
\end{equation}
where $\mathbf{\Phi}(t)=\int_{-\infty}^t \frac{1}{\sqrt{2\pi}} e^{-\frac{s^2}{2}} \, ds$ is the standard Gaussian distribution function.

Now consider the distribution function of the scaled random variable $\beta_p \widetilde{\mathcal{F}}_p^\phi(s_p)$, denoted by $\mathbf{G}_p$. By definition,
\begin{equation*}
\mathbf{G}_p(t) = \mathbb{P}_p\left(\beta_p \widetilde{\mathcal{F}}_p^\phi(s_p) \leq t\right)=\mathbb{P}_p\left(\widetilde{\mathcal{F}}_p^\phi(s_p) \leq \frac{t}{\beta_p}\right)=\mathbf{F}_p\left(\frac{t}{\beta_p}\right).
\end{equation*}

Since $\beta_p \to 1$, it follows that $\frac{t}{\beta_p} \to t$ for all $t \in \mathbb{R}$. Combining this with continuity of $\mathbf{\Phi}$ and (\ref{cltconv}), it follows that
\begin{equation*}
\mathbf{G}_p(t) \to \mathbf{\Phi}(t) \quad \text{as } p \to \infty, \text{ for all } t \in \mathbb{R}.
\end{equation*}

Thus, the sequence of random variables in (\ref{rewrite}) converges in distribution to $\mathcal{N}_{\mathbb{R}}(0,1)$, as required.
\end{proof}

\begin{proof}[Proof of Theorem \ref{maint}]
    Using the decomposition (\ref{Relation}), we write
    \begin{equation}
        \widehat{\mathcal{M}}_p = \widehat{\mathcal{M}}^{(1)}_p + \widehat{\mathcal{M}}^{(2)}_p,
    \end{equation}
    where
    \begin{equation*}
    \widehat{\mathcal{M}}_p = \frac{\mathcal{M}_p^\phi(s_p) - \mathbb{E}[\mathcal{M}_p^\phi(s_p)]}{\sqrt{\mathrm{Var}[\mathcal{M}_p^\phi(s_p)]}}, \quad
    \widehat{\mathcal{M}}^{(1)}_p = \frac{\mathcal{F}_p^\phi(s_p) - \mathbb{E}[\mathcal{F}_p^\phi(s_p)]}{\sqrt{\mathrm{Var}[\mathcal{M}_p^\phi(s_p)]}}, \quad
    \widehat{\mathcal{M}}^{(2)}_p = \frac{\mathcal{R}_p^\phi(s_p) - \mathbb{E}[\mathcal{R}_p^\phi(s_p)]}{\sqrt{\mathrm{Var}[\mathcal{M}_p^\phi(s_p)]}}.
    \end{equation*}

    Let $\mathbf{H}_p$, $\mathbf{H}_p^1$, and $\mathbf{\Phi}$ denote the distribution functions of $\widehat{\mathcal{M}}_p$, $\widehat{\mathcal{M}}^{(1)}_p$, and $Z \sim \mathcal{N}_\mathbb{R}(0,1)$, respectively. Fix $t \in \mathbb{R}$ and let $\epsilon > 0$. Then, since
    \begin{equation}
        \mathbf{H}_p(t) = \mathbb{P}_p\left(\widehat{\mathcal{M}}^{(1)}_p + \widehat{\mathcal{M}}^{(2)}_p \leq t\right).
    \end{equation}
    We can rewrite it as 
    \begin{equation}\label{decomp}
        \mathbf{H}_p(t) = \mathbb{P}_p\left(\widehat{\mathcal{M}}^{(1)}_p \leq t - \widehat{\mathcal{M}}^{(2)}_p\right) = \mathbb{P}_p(E_p^1) + \mathbb{P}_p(E_p^2),
    \end{equation}
    where
    \begin{equation*}
    E_p^1 = \left\{\widehat{\mathcal{M}}^{(1)}_p \leq t - \widehat{\mathcal{M}}^{(2)}_p\right\} \cap \left\{|\widehat{\mathcal{M}}^{(2)}_p| \geq \epsilon\right\}, \quad
    E_p^2 = \left\{\widehat{\mathcal{M}}^{(1)}_p \leq t - \widehat{\mathcal{M}}^{(2)}_p\right\} \cap \left\{|\widehat{\mathcal{M}}^{(2)}_p| < \epsilon\right\}.
    \end{equation*}

    For the first term in \eqref{decomp}, applying Chebyshev’s inequality, we get
    \begin{equation} \label{cheb}
        \mathbb{P}_p(E_p^1) \leq \mathbb{P}_p\left(|\widehat{\mathcal{M}}^{(2)}_p| \geq \epsilon\right) \leq \frac{1}{\epsilon^2} \frac{\mathrm{Var}[\mathcal{R}_p^\phi(s_p)]}{\mathrm{Var}[\mathcal{M}_p^\phi(s_p)]}.
    \end{equation}
    Using (\ref{remainder}) and (\ref{maincont}), it follows that
    \begin{equation} \label{varr}
    \frac{\mathrm{Var}[\mathcal{R}_p^\phi(s_p)]}{\mathrm{Var}[\mathcal{M}_p^\phi(s_p)]} \to 0 \quad \text{as } p \to \infty,
    \end{equation}
    and hence $\mathbb{P}_p(E_p^1) \to 0$ as $p\rightarrow \infty$.

    For the second term in \eqref{decomp}, note that
   \begin{equation}\label{bounds1}
        \mathbb{P}_p(E_p^2) \leq \mathbb{P}_p\left(\left\{\widehat{\mathcal{M}}^{(1)}_p \leq t + \epsilon\right\} \cap \left\{|\widehat{\mathcal{M}}^{(2)}_p| < \epsilon\right\}\right) \leq \mathbb{P}_p\left(\widehat{\mathcal{M}}^{(1)}_p \leq t + \epsilon\right) = \mathbf{H}_p^1(t + \epsilon),
   \end{equation}
    and
    \begin{equation}\label{bounds2}
        \mathbb{P}_p(E_p^2) \geq \mathbb{P}_p\left(\widehat{\mathcal{M}}^{(1)}_p \leq t - \epsilon\right) - \mathbb{P}_p\left(|\widehat{\mathcal{M}}^{(2)}_p| \geq \epsilon\right) 
        = \mathbf{H}_p^1(t - \epsilon) - \mathbb{P}_p\left(|\widehat{\mathcal{M}}^{(2)}_p| \geq \epsilon\right).
    \end{equation}
    From Lemma \ref{maincontclt}, we know that $\widehat{\mathcal{M}}^{(1)}_p \xrightarrow{d} Z \sim \mathcal{N}_{\mathbb{R}}(0,1)$. In particular, for any $\epsilon>0$
    \begin{equation*}
\mathbf{H}_p^1(t \pm \epsilon) \to \mathbf{\Phi}(t \pm \epsilon) \quad \text{as } p \to \infty.
    \end{equation*}
    On the other hand, by \eqref{cheb} and \eqref{varr} we have $\mathbb{P}_p\left(|\widehat{\mathcal{M}}^{(2)}_p| \geq \epsilon\right) \to 0$ as $p\rightarrow \infty.$ Combining these limits with the bounds in (\ref{bounds1})-(\ref{bounds2}), we conclude that
    \begin{equation*}
    \mathbf{\Phi}(t - \epsilon) \leq \liminf_{p \to \infty} \mathbb{P}_p(E_p^2) \leq \limsup_{p \to \infty} \mathbb{P}_p(E_p^2) \leq \mathbf{\Phi}(t + \epsilon).
    \end{equation*}
    Since $\epsilon > 0$ is arbitrary and $\mathbf{\Phi}$ is continuous, it follows that
    \begin{equation*}
    \mathbb{P}_p(E_p^2) \to \mathbf{\Phi}(t) \quad \text{as } p \to \infty.
    \end{equation*}
    Finally, returning to (\ref{decomp}) and using  $\mathbb{P}_p(E_p^1) \to 0$, we conclude that
    \begin{equation*}
    \mathbf{H}_p(t) \to \mathbf{\Phi}(t) \quad \text{as } p \to \infty.
    \end{equation*}
    Hence, the result follows.
\end{proof}

{}

\end{document}